\documentclass[11pt]{amsart}

\usepackage{graphicx}
\usepackage[english]{babel}
\usepackage[T1]{fontenc}
\usepackage[latin1]{inputenc}
\usepackage{amsfonts}
\usepackage{amssymb}
\usepackage{amsthm}
\usepackage{amsmath}
\usepackage{tikz}
\usepackage{float}
\usepackage{enumerate}
\usepackage{accents}
\usepackage{mathtools}
\usepackage{mathrsfs}
\usepackage{comment}
\usepackage{afterpage}
\usepackage{bbm}
\usepackage{dsfont}
\usepackage{stmaryrd}
\usepackage{hyperref}
\usepackage{mleftright}
\usepackage{subfig}
\usepackage{soul}
\usepackage{tikz-cd} 
\usepackage[foot]{amsaddr}

\theoremstyle{plain}
\newtheorem{thm}{Theorem}[section]
\newtheorem{lem}[thm]{Lemma}
\newtheorem{prop}[thm]{Proposition}
\newtheorem{cor}[thm]{Corollary}
\newtheorem*{thm*}{Theorem}
\newtheorem*{prop*}{Proposition}
\newtheorem*{cor*}{Corollary}

\theoremstyle{definition}
\newtheorem{defn}[thm]{Definition}

\newtheorem{rmk}[thm]{Remark}
\newtheorem*{rmk*}{Remark}

\newtheorem*{quest*}{Question}

\newtheorem*{defn*}{Definition}

\renewcommand{\o}{\circ}

\newcommand{\R}{\mathbb{R}}
\newcommand{\Z}{\mathbb{Z}}

\newcommand{\s}{\sigma}
\newcommand{\ra}{\rightarrow}

\newcommand{\LRa}{\Leftrightarrow}
\newcommand{\cu}{\subseteq}
\newcommand{\g}{\gamma}
\newcommand{\G}{\Gamma}

\newcommand{\mc}{\mathcal}
\newcommand{\mf}{\mathfrak}
\newcommand{\x}{\times}

\newcommand{\om}{\omega}
\newcommand{\id}{\mathrm{id}}
\newcommand{\acts}{\curvearrowright}

\newcommand{\mscr}{\mathscr}

\newcommand{\CAT}{{\rm CAT(0)}}

\DeclareMathOperator{\lk}{lk}
\DeclareMathOperator{\St}{st}
\DeclareMathOperator{\aut}{Aut}
\DeclareMathOperator{\is}{Isom}
\renewcommand{\L}{\Lambda}
\newcommand{\C}{\mc{C}}
\newcommand{\pf}{\pitchfork}

\begin{document}

\title[Automorphisms of contact graphs]{Automorphisms of contact graphs of \\ $\CAT$ cube complexes}
\author[E. Fioravanti]{Elia Fioravanti}\address{Max Planck Institute for Mathematics, Bonn, Germany}\email{fioravanti@mpim-bonn.mpg.de} 

\begin{abstract}
We show that, under weak assumptions, the automorphism group of a $\CAT$ cube complex $X$ coincides with the automorphism group of Hagen's contact graph $\C(X)$. The result holds, in particular, for universal covers of Salvetti complexes, where it provides an analogue of Ivanov's theorem on curve graphs of non-sporadic surfaces. This highlights a contrast between contact graphs and Kim--Koberda extension graphs, which have much larger automorphism group.

We also study contact graphs associated with Davis complexes of right-angled Coxeter groups. We show that these contact graphs are less well-behaved and describe exactly when they have more automorphisms than the universal cover of the Davis complex.
\end{abstract}

\maketitle

\section{Introduction}

The curve graph associated with a finite-type surface is a fundamental object in the study of mapping class groups \cite{Harvey1,Harvey2}. It has arguably been at the heart of the monumental developments in our understanding of mapping class groups and Kleinian groups that were initiated by the work of Masur and Minsky \cite{MM1,MM2} and have ultimately led to the solution of Thurston's ending lamination conjecture \cite{ending-1,ending-2}.  

An important property of curve graphs is that they are \emph{rigid}: Ivanov showed that every automorphism of the curve graph of a non-sporadic surface is induced by an element of the extended mapping class group \cite{Ivanov1,Ivanov2}. This was a central ingredient first in the computation of the abstract commensurator \cite{Ivanov2}, and later in the proof that mapping class groups are quasi-isometrically rigid \cite{BKMM,Ham-QI}. Due to the classical result of Tits that the automorphism group of a building coincides with the associated algebraic group \cite{Tits-build}, Ivanov's theorem also strengthens the analogy between mapping class groups and arithmetic groups.

It was recently shown by Behrstock, Hagen and Sisto \cite{HHS1,HHS2} that the Masur--Minsky machinery can be applied to a much vaster class of spaces, which they name \emph{hierarchically hyperbolic spaces}. This has proved a fruitful approach to a number of problems \cite{HHS-asdim,HHS-bound,HHS-dehn}, notably allowing these authors to obtain a particularly strong rigidity result for quasi-flats \cite{HHS-qf}.

It is natural to wonder if analogues of Ivanov's theorem hold for other hierarchically hyperbolic spaces. In this note, we address this question for $\CAT$ cube complexes. 

Many $\CAT$ cube complexes were shown to be hierarchically hyperbolic in \cite{HHS1,Hagen-Susse}, with Hagen's \emph{contact graph} playing the role of a curve complex. In particular, all virtually special groups \cite{HW-GAFA} are hierarchically hyperbolic, and these include all right-angled Artin and Coxeter groups. In the case of right-angled Artin groups, a parallel between curve graphs and \emph{extension graphs} (a relative of contact graphs) had already been observed by Kim and Koberda \cite{KK13,KK14}.

\medskip
For a general $\CAT$ cube complex $X$, there are actually at least three (closely related) graphs that can claim some analogy with curve graphs:
\begin{enumerate}
\item[(i)] Hagen's \emph{contact graph} $\C(X)$ \cite{Hagen-contact} mentioned above. Vertices are hyperplanes of $X$ and edges connect pairs of hyperplanes that are not separated by a third.
\item[(ii)] The \emph{crossing graph} $\C_{\pf}(X)$. This is the subgraph of $\C(X)$ with full vertex set and edges only joining pairs of transverse hyperplanes.
\item[(iii)] The \emph{reduced crossing graph} $\C_r(X)$. This is the simplicial graph obtained  by identifying vertices of $\C_{\pf}(X)$ with the same link (see Definition~\ref{C_r defn} for a more precise description). 

When $X_{\G}$ is the universal cover of a Salvetti complex, and if no two vertices of the graph $\G$ have the same link, then $\C_r(X_{\G})$ coincides with the \emph{extension graph} $\G^e$ introduced by Kim and Koberda \cite{KK13}.
\end{enumerate}

If $X$ is uniformly locally finite and has no cut vertices, then these three graphs are quasi-isometric to each other (see e.g.\ \cite[Appendix~A]{Genevois-surv}). By \cite[Theorem~4.1]{Hagen-contact}, they are quasi-trees, hence, in particular, they are $\delta$--hyperbolic. When, in addition, $X$ is hierarchically hyperbolic, a number of further analogies with curve graphs and mapping class groups is discussed in \cite{HHS1}, including acylindricity of actions, existence of hierarchy paths, and a Masur--Minsky-style distance formula. We also refer the reader to \cite{KK14} for other analogies in the case of right-angled Artin groups.

\medskip
If $\mc{G}(X)$ denotes any of the three graphs above, we have a natural homomorphism $\aut X\ra\aut\mc{G}(X)$. The closest we can get to an analogue of Ivanov's theorem is if one of these homomorphisms is an isomorphism. It is important to remark that the group $\aut X$ will often be \emph{uncountable} and, in fact, this is essentially always the case for universal covers of Salvetti complexes (see Remark~\ref{non-discrete} below).

It can be deduced from the work of Huang \cite{Huang-GT2} that the images of the homomorphisms $\aut X\ra\aut\C_{\pf}(X)$ and $\aut X\ra\aut\C_r(X)$ have uncountable index for most universal covers of Salvetti complexes (see Remark~\ref{other graphs don't work} below). This dashes any hopes that these two maps be isomorphisms, even for relatively harmless cube complexes like Salvetti complexes. 

As we are about to see, things are a lot better behaved for the contact graph $\C(X)$. We only draw the reader's attention to Example~7.1 in \cite{HHS1}, which shows that even the homomorphism $\iota\colon\aut X\ra\aut\C(X)$ cannot be an isomorphism in complete generality. 

\medskip
We say that a vertex $v\in X^{(0)}$ is \emph{extremal} if its link is a cone (cf.\ \cite[Definition~2.2]{BFI}). In other words, $v$ lies in the carrier of some hyperplane $\mf{w}$ that is transverse to all other hyperplanes containing $v$ in their carrier. We stress that any vertex that belongs to a single edge of $X$ is extremal.

Every hyperplane $\mf{w}\cu X$ inherits a structure of $\CAT$ cube complex from $X$, where cubes of $\mf{w}$ are intersections with cubes of $X$. In particular, vertices of $\mf{w}$ are in one-to-one correspondence with edges of $X$ crossing $\mf{w}$. It thus makes sense to speak of \emph{extremal vertices of $\mf{w}$}.

Our main result is the following (its two parts will be proved in Corollary~\ref{the map rho} and Theorem~\ref{isomorphism}, respectively). We denote by $S(X^{(0)})$ the group of permutations of the vertex set of $X$.

\begin{thm*} \hypertarget{main intro}{} 
Let $X$ be a uniformly locally finite $\CAT$ cube complex with no extremal vertices. Let $\iota\colon\aut X\ra\aut\C(X)$ be as above.
\begin{enumerate}
\item There exists a homomorphism $\rho\colon\aut\C(X)\ra S(X^{(0)})$ such that $\rho\o\iota=\id_{\aut X}$. In particular, the homomorphism $\iota$ is injective.
\end{enumerate}
Suppose in addition that no hyperplane of $X$ has extremal vertices. Then:
\begin{enumerate}
\setcounter{enumi}{1}
\item $\iota$ is a group isomorphism with inverse $\rho$.
\end{enumerate}
\end{thm*}

The main idea in the proof of the \hyperlink{main intro}{Theorem} is that, when $X$ has no extremal vertices, vertices of $X$ are in one-to-one correspondence with maximal cliques in $\C(X)$. If, in addition, the hyperplanes of $X$ have no extremal vertices, edges of $X$ correspond to pairs of maximal cliques of $\C(X)$ with the largest possible intersection.

Cube complexes with no free faces never have extremal vertices (see e.g.\ \cite[Remark~2.5]{BFI}), and neither do their hyperplanes. Hence:

\begin{cor*} 
Let $X$ be a uniformly locally finite $\CAT$ cube complex with no free faces. Then $\aut X\cong\aut\C(X)$ via the map $\iota$.
\end{cor*}

The Corollary applies in particular to all universal covers of Salvetti complexes associated with right-angled Artin groups.

It would be nice to use this result to study abstract commensurators of right-angled Artin groups, or to expand the known results on their quasi-isometry classification \cite{Behrstock-Neumann,BJN,Huang-GT2,Margolis-RAAG} and their quasi-i\-so\-met\-ric rigidity properties \cite{BKS,Huang-Kleiner,Huang-IM}. Unfortunately, it appears that these applications would rather require the \emph{extension graph}, which, as discussed above, is highly non-rigid.

For instance, Huang showed that every quasi-isometry between right-angled Artin groups with finite outer automorphism groups induces an isomorphism of their extension graphs \cite[Lemma~4.5]{Huang-GT2}. In the same context, quasi-isometries can fail to induce isomorphisms of contact graphs (see e.g.\ Remark~\ref{QIs incompatible with C}). 

As pointed out by one of the anonymous referees, it is even true that every right-angled Artin group $G$ with $\#{\rm Out}(G)<+\infty$ admits a proper cobounded quasi-action of some group $H$ that fails to induce an $H$--action by automorphisms on the contact graph associated with the Salvetti complex of $G$ (see the proof of \cite[Theorem 6.10]{Huang-Kleiner} along with our \hyperlink{main intro}{Theorem}). It remains unclear to me whether there can be such examples where this issue cannot be resolved by passing to a finite-index subgroup of $H$. 

In relation to this, the following question was suggested by the same anonymous referee (several special cases have been studied in \cite{Huang-IM}):

\begin{quest*} 
Let a group $H$ be quasi-isometric to a right-angled Artin group $G$ with $\#{\rm Out}(G)<+\infty$. Does a finite-index subgroup of $H$ act properly and cocompactly on the universal cover of the Salvetti complex of $G$? 
\end{quest*}

Finally, it is reasonable to wonder whether it really is necessary to require that hyperplanes of $X$ have no extremal vertices in the \hyperlink{main intro}{Theorem}. Davis complexes provide a nice class of counterexamples when we drop this hypothesis. 

Recall that, for a graph $\G$ and a vertex $a\in\G^{(0)}$, the \emph{star} $\St a\cu\G^{(0)}$ is the set of vertices that are either equal to $a$ or joined to $a$ by an edge. The following will be proved in Proposition~\ref{prop recalled}.

\begin{prop*} \hypertarget{Davis intro}{} 
Let $W_{\G}$ be a right-angled Coxeter group with no finite direct factors. Let $Y_{\G}$ denote the universal cover of its Davis complex. Then:
\begin{enumerate}
\item $Y_{\G}$ has no extremal vertices, so $\iota\colon\aut Y_{\G}\ra\aut\C(Y_{\G})$ is injective.
\item $Y_{\G}$ has a hyperplane with extremal vertices if and only if there exist distinct vertices $a,b\in\G^{(0)}$ with $\St a\cu\St b$. In this case, the subgroup $\iota(\aut Y_{\G})<\aut\C(Y_{\G})$ has infinite index.
\item The homomorphism $\rho\colon\aut\C(Y_{\G})\ra S(Y_{\G}^{(0)})$ is injective if and only if there do not exist vertices $a,b\in\G^{(0)}$ with $\St a=\St b$. When $\rho$ is not injective, its kernel is an uncountable torsion subgroup.
\end{enumerate}
\end{prop*}

\medskip
{\bf Acknowledgments.} I am grateful to Max Planck Institute for Mathematics in Bonn for its hospitality and financial support. 

I would like to thank Nir Lazarovich for pointing out that universal covers of Salvetti complexes have uncountable automorphism group (cf.\ Remark~\ref{non-discrete}) and directing me to \cite{Haglund-Paulin}. I would also like to thank Jason Behrstock for some comments on an earlier preprint, Mark Hagen for interesting conversations related to this work, and the anonymous referees for their many useful suggestions.

\section{Preliminaries and examples.}\label{prelims}

All proofs are elementary, but we assume a certain familiarity with the basics of $\CAT$ cube complexes. See for instance \cite{Wise-astra,Sag14} for an introduction.

\medskip
Given a graph $\mc{G}$, we say that two vertices are \emph{adjacent} if they belong to a common edge. Given $x\in\mc{G}^{(0)}$, the \emph{link} $\lk x\cu\mc{G}^{(0)}$ is the subset of vertices adjacent to $x$. The \emph{star} of $x$ is the set $\St x=\lk x\cup\{x\}$. 

An \emph{automorphism} of $\mc{G}$ is a self-bijection of $\mc{G}^{(0)}$ that preserves adjacency of vertices. We denote the automorphism group of $\mc{G}$ by $\aut\mc{G}$. If every vertex in $\mc{G}$ has finite degree, $\aut\mc{G}$ is a second-countable, locally compact topological group with the compact-open topology,

If $\mc{G}_1$ and $\mc{G}_2$ are graphs, their \emph{join} $\mc{G}_1\ast\mc{G}_2$ is obtained by taking the disjoint union $\mc{G}_1\sqcup\mc{G}_2$ and joining by an edge each vertex of $\mc{G}_1$ to each vertex of $\mc{G}_2$. We say that the graph $\mc{G}$ is a \emph{cone} if it is the join of some other graph $\mc{G}'$ and a singleton.

\medskip
Let $X$ be a $\CAT$ cube complex. The group of \emph{(cubical) automorphisms} of $X$ coincides the automorphism group of the $1$--skeleton $X^{(1)}$; we denote it by $\aut X$. A path $\g\cu X$ is a \emph{combinatorial geodesic} if it is contained in $X^{(1)}$ and it is a geodesic for the graph metric on $X^{(1)}$.

We denote by $\mscr{W}(X)$ and $\mscr{H}(X)$, respectively, the set of all hyperplanes and all halfspaces of $X$. Given a hyperplane $\mf{w}\in\mscr{W}(X)$, we refer to the two halfspaces $\mf{h},\mf{h}^*$ bounded by $\mf{w}$ as its \emph{sides}. Given subsets $A,B\cu X$, we denote by $\mscr{W}(A|B)$ the set of hyperplanes $\mf{w}$ such that $A$ and $B$ are contained in opposite sides of $\mf{w}$.

Given $\mf{w}\in\mscr{W}(X)$, the union of all cubes of $X$ that intersect $\mf{w}$ forms a subcomplex $C(\mf{w})\cu X$ known as the \emph{carrier} of $\mf{w}$. We say that $\mf{w}$ is \emph{adjacent} to a vertex $v\in X^{(0)}$ if $\mf{w}$ intersects an edge of $X$ incident to $v$ (equivalently, if $v$ belongs to the carrier of $\mf{w}$). We denote by $\mscr{W}_v\cu\mscr{W}(X)$ the set of hyperplanes adjacent to the vertex $v\in X^{(0)}$.

We say that a hyperplane $\mf{u}$ \emph{contacts} another hyperplane $\mf{w}$ if their carriers intersect. Equivalent conditions are that the set $\mscr{W}(\mf{u}|\mf{w})$ is empty, or that there exists a vertex $v\in X^{(0)}$ such that $\{\mf{u},\mf{w}\}\cu\mscr{W}_v$. 

For each $v\in X^{(0)}$, we redefine\footnote{As a set, $\lk v$ is naturally identified with the link within the graph $X^{(1)}$. However, when $v$ is a vertex of a cube complex, it will be important that $\lk v$ has an additional structure of graph.} the link $\lk v$ as follows. This is the graph that has a vertex for every edge of $X$ incident to $v$, and an edge joining two vertices of $\lk v$ if and only if the corresponding edges of $X$ span a square. Equivalently, the vertex set of $\lk v$ is $\mscr{W}_v$, and we join two hyperplanes by an edge when they are transverse. We say that $v$ is \emph{extremal} if $\lk v$ is a cone.

The intersection between a hyperplane $\mf{w}\cu X$ and a cube $c\cu X$ is always either empty or a mid-cube in $c$. It follows that $\mf{w}$ inherits a decomposition into cubes $\mf{w}\cap c$, where $c$ ranges through all cubes in the carrier $C(\mf{w})$. This gives $\mf{w}$ a structure of $\CAT$ cube complex. 
We are thus allowed to speak of ``vertices of $\mf{w}$'' (which are midpoints of edges of $X$ crossing $\mf{w}$) and of their link in the $\CAT$ cube complex $\mf{w}$.

A subcomplex $C\cu X$ is said to be \emph{convex} if every combinatorial geodesic joining two vertices of $C$ is entirely contained in $C$. The carrier of every hyperplane is a convex subcomplex, and so is the full subcomplex spanned by the vertices contained in any given halfspace. If $C_1,\dots,C_k$ are pairwise-intersecting convex subcomplexes of $X$, we have $C_1\cap\dots\cap C_k\neq\emptyset$. This fact normally goes by the name of Helly's lemma.

Before we go any further, let us give a less ambiguous definition of the \emph{reduced crossing graph} $\C_r(X)$, which was sketched in the introduction.

\begin{defn}\label{C_r defn}
Vertices of $\C_r(X)$ are maximal subsets $\mc{W}\cu\mscr{W}(X)$ with the property that, for any two $\mf{u},\mf{v}\in\mc{W}$, each hyperplane of $X$ is transverse to $\mf{u}$ if and only if it is transverse to $\mf{v}$. Vertices of $\C_r(X)$ corresponding to subsets $\mc{U},\mc{V}\cu\mscr{W}(X)$ are joined by an edge of $\C_r(X)$ if and only if there exist transverse hyperplanes $\mf{u}\in\mc{U}$ and $\mf{v}\in\mc{V}$ (in which case every hyperplane in $\mc{U}$ is transverse to every hyperplane in $\mc{V}$).
\end{defn}

We now obtain Remarks~\ref{non-discrete},~\ref{other graphs don't work} and~\ref{QIs incompatible with C}, which were promised in the introduction. Throughout this discussion, we denote by $X_{\G}$ the universal cover of the Salvetti complex associated with a right-angled Artin group $A_{\G}$.

Recall that the $1$--skeleton of $X_{\G}$ coincides with the standard Cayley graph of $A_{\G}$. It follows that each edge of $X_{\G}$ is labelled by a standard generator of $A_{\G}$, i.e.\ a vertex of the graph $\G$. Opposite edges in a square of $X_{\G}$ always have the same label, hence all edges crossing a given hyperplane $\mf{w}\in\mscr{W}(X_{\G})$ have the same label. We will refer to this element of $\G^{(0)}$ as the \emph{label} of $\mf{w}$.  

\begin{rmk}\label{non-discrete}
It was pointed out to me by Nir Lazarovich that the group $\aut X_{\G}$ is always uncountable, except when $X_{\G}\cong\R^n$. The argument is essentially the one in Theorem~5.12 of \cite{Haglund-Paulin}, but I briefly recall it here.

Assume that the graph $\G$ is not complete and pick vertices $x,y\in\G^{(0)}$ that are not joined by an edge. Let $\phi$ be the automorphism of the group $A_{\G}$ that fixes each standard generator except for the one corresponding to $y$, which is taken to its inverse. Identifying $A_{\G}$ with $X_{\G}^{(0)}$, it is clear that $\phi$ determines a cubical automorphism of $X_{\G}$, which we also denote by $\phi$.

Let $v\in X_{\G}^{(0)}$ be the vertex corresponding to the identity of $A_{\G}$. There are two hyperplanes adjacent to $v$ that are labelled by $x\in\G^{(0)}$. Let $\mf{w}$ be one of them and let $\mf{h},\mf{h}^*$ denote its two sides. Observe that $\phi$ fixes the carrier $C(\mf{w})$ pointwise, and therefore leaves invariant $\mf{h}$ and $\mf{h}^*$.

Let the map $\psi\colon X_{\G}\ra X_{\G}$ be defined as the identity on $\mf{h}^*\cup C(\mf{w})$, and as $\phi$ on $\mf{h}\cup C(\mf{w})$. It follows from the previous observation that $\psi$ is well-defined. Since every edge of $X_{\G}$ is contained in either $\mf{h}^*\cup C(\mf{w})$ or $\mf{h}\cup C(\mf{w})$, this is an automorphism of $X_{\G}$. It is clear that $\psi\neq\id_{X_{\G}}$.

Now, given any finite set $F$ of vertices of  $X_{\G}$, there exists $g\in A_{\G}$ with $gF\cu\mf{h}^*$. Hence the automorphism $g^{-1}\psi g\in\aut X_{\G}-\{\id_{X_{\G}}\}$ fixes $F$ pointwise. We conclude that the locally compact group $\aut X_{\G}$ is non-discrete.

The fact that $\aut X_{\G}$ is uncountable follows from Baire's theorem (see e.g.\ \cite[Remark~2.A.18]{Cornulier-Harpe}).
\end{rmk}

We say that a combinatorial geodesic $\g\cu X_{\G}$ is \emph{standard} if all edges of $\g$ have the same label. Two bi-infinite standard geodesics are at finite Hausdorff distance if and only if they cross the same hyperplanes; in this case, we say that they are \emph{parallel}. Given that $X_{\G}^{(1)}$ coincides with the usual Cayley graph of $A_{\G}$, we will also speak of standard geodesics in $A_{\G}$. 

The \emph{extension graph} $\G^e$ \cite{KK13,KK14,Huang-GT2} has a vertex for every parallelism class of standard geodesics in $X_{\G}$; the vertices determined by standard geodesics $\g_1$ and $\g_2$ are joined by an edge of $\G^e$ if and only if the hyperplanes crossed by $\g_1$ are transverse to the hyperplanes crossed by $\g_2$. 

Note that, in general, we do not have a homomorphism $\aut X_{\G}\ra\aut\G^e$, but only a homomorphism $A_{\G}\ra\aut\G^e$.

Let $d_w$ denote the usual word metric on $A_{\G}$, which coincides with the graph metric on $X_{\G}^{(1)}$ under the identification $A_{\G}=X_{\G}^{(0)}$. Let $d_r$ be the \emph{syllable metric} on $A_{\G}$, as defined e.g.\ in \cite[Section~5.2]{KK14} and \cite[Section~4.3]{Huang-GT2}. More precisely, $d_r$ is the largest metric on $A_{\G}$ satisfying $d_r(x,y)=1$ for all distinct $x,y\in A_{\G}$ that are joined by a standard geodesic.

In order to make Remark~\ref{other graphs don't work} below, we will need the following lemma, which can in large part be deduced from the work of Huang \cite{Huang-GT2}.

\begin{lem}\label{lem for rmk}
Let $\G$ be a finite graph.
\begin{enumerate}
\item There is a natural homomorphism $\is (A_{\G},d_r)\ra\aut\G^e$. This is injective if and only if $\G$ is not a cone\footnote{Note that this requirement seems to have been overlooked in Remark~4.16 of \cite{Huang-GT2}.}.
\item If no two vertices of $\G$ have the same link, every isometry of $(A_{\G},d_w)$ is an isometry of $(A_{\G},d_r)$. Moreover, $\G^e=\C_r(X_{\G})$ in this case.
\end{enumerate}
\end{lem}
\begin{proof}
We begin with part~(2). Consider a vertex $v\in X_{\G}^{(0)}$, a vertex $x\in\G^{(0)}$, and the two vertices $x^{\pm}\in(\lk v)^{(0)}$ determined by $x$. If no two vertices of $\G$ have the same link, no vertex of $\lk v-\{x^{\pm}\}$ can have the same link as $x^+$ and $x^-$. It follows that every element of $\aut X_{\G}$ takes standard geodesics to standard geodesics. Identifying $\is (A_{\G},d_w)$ with $\aut X_{\G}$, we deduce that every $\phi\in\is (A_{\G},d_w)$ is $1$--Lipschitz with respect to $d_r$. Since $\phi^{-1}$ must also be $1$--Lipschitz for $d_r$, this shows that $\phi\in\is (A_{\G},d_r)$.

Finally, it is easy to see that, since no two vertices of $\G$ have the same link, the projection $\C_{\pf}(X_{\G})\twoheadrightarrow\C_r(X_{\G})$ identifies two vertices if and only if there exists a standard geodesic crossing the corresponding hyperplanes. It follows that $\C_r(X_{\G})$ coincides with $\G^e$ in this case.

Regarding part~(1), Huang showed that every element of $\is (A_{\G},d_r)$ takes standard geodesics to standard geodesics, as unparametrised sets (see Remark~4.16 and the proof of Corollary~4.15 in \cite{Huang-GT2}). Since vertices of $\G^e$ correspond to families of hyperplanes transverse to standard geodesics, every isometry of the syllable metric induces a permutation of the vertices of $\G^e$. Huang also showed (\emph{loc.\ cit.}) that such permutations preserve adjacency of vertices of $\G^e$. We thus obtain a homomorphism $\is (A_{\G},d_r)\ra\aut\G^e$.

If $\G$ is the cone over a subgraph $\Delta$, we have $A_{\G}=A_{\Delta}\x\Z$. For any permutation $\s\colon\Z\ra\Z$, the map $(g,n)\mapsto (g,\s(n))$ is an isometry of the syllable metric on $A_{\Delta}\x\Z$, but it maps to the identity in $\aut\G^e$. 

Conversely, let us show that, when $\G$ is not a cone, the homomorphism $\is (A_{\G},d_r)\ra\aut\G^e$ is injective. In other words, given $\phi\in\is (A_{\G},d_r)$ taking each standard geodesic to a standard geodesic at finite Hausdorff distance, we need to show that $\phi$ fixes each element of $A_{\G}$. 

Consider an element $g\in A_{\G}$ and let $\{\g_x\}_{x\in\G^{(0)}}$ be the collection of all standard geodesics containing $g$. Observe that $\bigcap_x\g_x=\{g\}$, that each $\phi(\g_x)$ is a standard geodesic at finite Hausdorff distance from $\g_x$, and that $\bigcap_x\phi(\g_x)=\{\phi(g)\}$. Let $\alpha$ be a combinatorial geodesic joining $g$ and $\phi(g)$ in $X_{\G}$. Since $\alpha$ joins a point of $\g_x$ to a point of $\phi(\g_x)$, every edge crossed by $\alpha$ must be labelled by an element of $\St x\cu\G^{(0)}$. However, since $\G$ is not a cone, we have $\bigcap_{x\in\G^{(0)}}\St x=\emptyset$. In conclusion, $\alpha$ does not cross any edges, and we have $\phi(g)=g$ for all $g\in A_{\G}$.
\end{proof}

\begin{rmk}\label{other graphs don't work}
Suppose that $\G$ is not a cone and that no two vertices of $\G$ have the same link. We show here that, in this case, the images of the two natural maps $\iota_{\pf}\colon\aut X_{\G}\ra\aut\C_{\pf}(X_{\G})$ and $\iota_r\colon\aut X_{\G}\ra\aut\C_r(X_{\G})$ have uncountable index.

It follows from Lemma~\ref{lem for rmk} that we have a commutative diagram:
\[ 
\begin{tikzcd}
\is (A_{\G},d_w) \arrow[hookrightarrow]{r} & \is (A_{\G},d_r) \arrow[hookrightarrow]{r} & \aut \G^e \\
\aut X_{\G} \arrow[u,equals] \arrow[hookrightarrow]{rr}{\iota_r} & & \aut\C_r(X_{\G}). \arrow[u,equals]
\end{tikzcd}
\]
Now, the argument in \cite[Example~4.14]{Huang-GT2} shows that the embedding $\is (A_{\G},d_w)\hookrightarrow\is (A_{\G},d_r)$ is very far from being surjective. More precisely, for every standard geodesic $\g\cu X_{\G}$ and every permutation $\s$ of its vertex set $\g^{(0)}\cu X_{\G}^{(0)}=A_{\G}$, we can construct an element of $\is (A_{\G},d_r)$ that preserves the set $\g^{(0)}$, and acts on it as $\s$. 

On closer inspection, this corresponds to a copy of the infinite symmetric group $S(\Z)<\is (A_{\G},d_r)$ that intersects the subgroup $\is (A_{\G},d_w)$ in an infinite dihedral subgroup. Thus, $\is (A_{\G},d_w)<\is (A_{\G},d_r)$ has uncountable index, and so does $\aut X_{\G}<\aut\C_r(X_{\G})$.

Finally, observe that the quotient projection $\C_{\pf}(X_{\G})\twoheadrightarrow\C_r(X_{\G})$ induces a surjective\footnote{This holds for arbitrary $\CAT$ cube complexes $X$, as long as the fibres of the projection $\C_{\pf}(X)\twoheadrightarrow\C_r(X)$ all have the same cardinality.} homomorphism $\pi_r\colon\aut\C_{\pf}(X_{\G})\twoheadrightarrow\aut\C_r(X_{\G})$. This yields the commutative diagram:
\[ 
\begin{tikzcd}
\aut X_{\G}  \arrow[hookrightarrow]{rr}{\iota_{\pf}} \arrow[hookrightarrow]{dr}{\iota_r} &  & \aut\C_{\pf}(X_{\G}) \arrow[twoheadrightarrow,swap]{dl}{\pi_r} \\
 & \aut\C_r(X_{\G}). &
 \end{tikzcd}
\]
We conclude that $\aut X_{\G}<\aut\C_{\pf}(X_{\G})$ also has uncountable index.
\end{rmk}

\begin{rmk}\label{QIs incompatible with C}
There are many examples of quasi-isometries of $X_{\G}$ that fail to induce an automorphism of the contact graph $\C_{\G}:=\mc{C}(X_{\G})$.

A first example is provided by the argument in Remark~\ref{other graphs don't work} (also appearing in \cite[Section~11]{BKS} and \cite[Example~4.14]{Huang-GT2}). Under its assumptions, we have seen that every permutation of the vertex set of the standard geodesic $\g$ gives rise to an isometry of $(A_{\G},d_r)$. Considering bi-Lipschitz permutations of $\g^{(0)}$, the resulting isometries of $d_r$ are quasi-isometries of $(A_{\G},d_w)$ and most of them will not induce automorphisms of $\C_{\G}$.

For a second example, suggested by one of the anonymous referees, let $\Lambda$ be a pentagon. Consider the homomorphism $\tau\colon A_{\Lambda}\ra\Z/2\Z$ with $\tau(v)=1$ for some vertex $v\in\Lambda$ and $\tau(w)=0$ for all other $w\in\Lambda$. Then $\ker\tau$ is isomorphic to $A_{\Lambda'}$, where $\Lambda'$ is the graph obtained by glueing two copies of $\Lambda$ along $\St v$. Since $\ker\tau$ has index $2$ in $A_{\Lambda}$, there is a natural proper cobounded quasi-action of $A_{\Lambda}$ on $X_{\Lambda'}$. It is not hard to see that every hyperplane of $X_{\Lambda'}$ is taken within bounded distance of another hyperplane, so we obtain an action of $A_{\Lambda}$ on $\mc{C}_{\Lambda'}^{(0)}$ by bijections. However, the bijection associated with $v\in A_{\Lambda}$ does not extend to an isomorphism of $\mc{C}_{\Lambda'}$.

Note that it is not clear if the quasi-isometries in the first example can take part in a proper cobounded quasi-action on $X_{\G}$. In the second example, we do have a geometric quasi-action, but $\#{\rm Out}(A_{\Lambda'})=+\infty$. Moreover, the restriction to $\ker\tau<A_{\Lambda}$ of the quasi-action on $A_{\Lambda'}$ does in fact induce an action by automorphisms on $\mc{C}_{\Lambda'}$. As mentioned in the Introduction, other examples are provided by \cite[Theorem~6.10]{Huang-Kleiner}, although, again, all issues disappear in a finite-index subgroup.
\end{rmk}

\section{Proof of the \protect\hyperlink{main intro}{Theorem}.}

Let $X$ be a $\CAT$ cube complex with contact graph $\C=\C(X)$, as defined in the introduction. We identify subsets of $\C$ with their intersection with $\C^{(0)}$ and with the corresponding subset of $\mscr{W}(X)$.

Our first goal is to establish a correspondence between vertices of $X$ and maximal cliques in $\C$.

\begin{lem}\label{structure of cliques}
\begin{enumerate}
\item[]
\item For every finite clique $C\cu\C$, there exists $v\in X^{(0)}$ with $C\cu\mscr{W}_v$.
\item If $C\cu\C$ is a maximal finite clique, there is $v\in X^{(0)}$ with $C=\mscr{W}_v$. 
\item If $X$ is uniformly locally finite, the cliques of $\mc{C}$ are uniformly finite.
\end{enumerate}
\end{lem}
\begin{proof}
For every vertex $v\in X^{(0)}$, the subset $\mscr{W}_v\cu\C$ is a clique. Parts~(2) and~(3) thus follow immediately from part~(1), which we now prove.

Let $C\cu\C$ be a finite clique. For every hyperplane $\mf{w}\in C$, at most one side of $\mf{w}$ can contain a hyperplane in $C$ disjoint from $\mf{w}$. Picking this side for every $\mf{w}\in C$, or just any side if $\mf{w}$ is transverse to all other hyperplanes in $C$, we obtain a finite collection of pairwise-intersecting halfspaces $\mc{H}\cu\mscr{H}(X)$. By Helly's lemma, there exists $w\in X^{(0)}$ lying in all elements of $\mc{H}$. 

Let $d(w)$ denote the sum of the distances from $w$ to the carriers of the hyperplanes in $C$, using the graph metric of $X^{(1)}$. Thus, $d(w)=0$ if and only if $C\cu\mscr{W}_v$. If $d(w)>0$, there exist hyperplanes $\mf{u}\in C-\mscr{W}_w$ and $\mf{v}\in\mscr{W}_w\cap\mscr{W}(w|\mf{u})$. Let $w'\in X^{(0)}$ be the vertex with $\mscr{W}(w|w')=\{\mf{v}\}$. No hyperplane in $C$ can be contained in the side of $\mf{v}$ that contains $w$, or they would not contact $\mf{u}$. It follows that $d(w')<d(w)$ and, iterating this procedure finitely many times, we obtain a vertex $v\in X^{(0)}$ with $d(v)=0$. This yields part~(1). 
\end{proof}

\begin{rmk}
Locally finite $\CAT$ cube complexes need not be $\om$--di\-men\-sio\-nal (i.e.\ they can contain infinite families of pairwise-transverse hyperplanes). In particular, there exist locally finite cube complexes whose contact, crossing, and reduced crossing graphs all contain infinite cliques. Thus, the hypothesis in part~(3) of Lemma~\ref{structure of cliques} cannot be weakened.
\end{rmk}

\begin{lem}\label{cone links}
Consider a vertex $v\in X^{(0)}$ with $\#\mscr{W}_v<+\infty$.
\begin{enumerate}
\item There exists $w\neq v$ with $\mscr{W}_v\cu\mscr{W}_w$ if and only if $\lk v$ is a cone.
\item In particular, if $\lk v$ is not a cone, the clique $\mscr{W}_v\cu\C$ is maximal and there does not exist another vertex $w$ with $\mscr{W}_v=\mscr{W}_w$.
\end{enumerate}
\end{lem}
\begin{proof}
By part~(1) of Lemma~\ref{structure of cliques}, $\mscr{W}_v$ is a maximal clique if and only if there does not exist $w\in X^{(0)}$ with $\mscr{W}_v\subsetneq\mscr{W}_w$. Part~(2) thus follows from part~(1).

Suppose that $w\neq v$ is a vertex with $\mscr{W}_v\cu\mscr{W}_w$. Let $\mf{w}\in\mscr{W}_v$ be a hyperplane separating $w$ and $v$. Since $\mf{w}$ cannot separate $w$ from any element of $\mscr{W}_v$, it must be transverse to all elements of $\mscr{W}_v$. Hence $\lk v$ is a cone.

Conversely, if $\lk v$ is a cone, there exists a hyperplane $\mf{w}\in\mscr{W}_v$ that is transverse to all other hyperplanes adjacent to $v$. Denoting by $w\in X^{(0)}$ the vertex with $\mscr{W}(v|w)=\{\mf{w}\}$, we have $\mscr{W}_v\cu\mscr{W}_w$.
\end{proof}

Recall from the introduction that the action $\aut X\acts\mscr{W}(X)$ results in a natural homomorphism $\iota\colon\aut X\ra\aut\C$. Lemmas~\ref{structure of cliques} and~\ref{cone links} immediately yield the first part of the \hyperlink{main intro}{Theorem}:

\begin{cor}\label{the map rho}
Let $X$ be uniformly locally finite, with no extremal vertices. There exists a natural one-to-one correspondence between vertices of $X$ and maximal cliques of $\mc{C}$. This induces a homomorphism $\rho\colon\aut\C\ra S(X^{(0)})$ satisfying $\rho\o\iota=\id_{\aut X}$.
\end{cor}

We now proceed to study when the homomorphism $\rho$ is injective.

\begin{defn}\label{I(w) defn}
For a hyperplane $\mf{w}\in\mscr{W}(X)$, we denote by $\mc{I}(\mf{w})$ the intersection of all sets $\mscr{W}_v$ that contain $\mf{w}$. Let moreover $\mc{I}^0(\mf{w})\cu\mc{I}(\mf{w})$ be the subset of those hyperplanes $\mf{u}\in\mc{I}(\mf{w})$ for which $\mf{w}\in\mc{I}(\mf{u})$.
\end{defn}

Note that $\mf{w}\in\mc{I}^0(\mf{w})\cu\mc{I}(\mf{w})$. Recall that a vertex $v\in X^{(0)}$ lies in the carrier of $\mf{w}$ if and only if $\mf{w}\in\mscr{W}_v$. Thus, we have $\mf{u}\in\mc{I}(\mf{w})$ if and only if the carrier of $\mf{w}$ is contained in the carrier of $\mf{u}$. In particular, $\mf{u}\in\mc{I}^0(\mf{w})$ if and only if $\mf{u}$ and $\mf{w}$ have the same carrier. 

\begin{rmk}\label{I(w) and automorphisms}
Let $X$ be uniformly locally finite, with no extremal vertices. Since the subsets $\mscr{W}_v\cu\mscr{W}(X)$ are exactly the maximal cliques of $\C$, we have $\phi(\mc{I}(\mf{w}))=\mc{I}(\phi(\mf{w}))$ for all $\mf{w}\in\mscr{W}(X)$ and $\phi\in\aut\C$. Moreover:
\begin{align*}
\mf{u}\in\mc{I}^0(\mf{w}) &\LRa \mf{u}\in\mc{I}(\mf{w}) \text{ \& } \mf{w}\in\mc{I}(\mf{u}) \\
&\LRa \phi(\mf{u})\in\phi(\mc{I}(\mf{w})) \text{ \& } \phi(\mf{w})\in\phi(\mc{I}(\mf{u})) \\
&\LRa \phi(\mf{u})\in\mc{I}(\phi(\mf{w})) \text{ \& } \phi(\mf{w})\in\mc{I}(\phi(\mf{u})) \LRa \phi(\mf{u})\in\mc{I}^0(\phi(\mf{w})).
\end{align*}
Hence $\phi(\mc{I}^0(\mf{w}))=\mc{I}^0(\phi(\mf{w}))$ as well.
\end{rmk}

\begin{lem}\label{I(w) lemma}
Consider $\mf{w},\mf{u}\in\mscr{W}(X)$.
\begin{enumerate}
\item We have $\mf{u}\in\mc{I}(\mf{w})$ if and only if $\mf{u}$ is transverse to $\mf{w}$ and to all other hyperplanes transverse to $\mf{w}$. In particular, $\#\mc{I}(\mf{w})\leq\dim X$.
\item If $\mf{u}\in\mc{I}^0(\mf{w})$, then $\mf{u}$ and $\mf{w}$ have exactly the same star in $\C$.
\item We have $\mf{u}\in\mc{I}^0(\mf{w})$ if and only if $\mc{I}^0(\mf{u})=\mc{I}^0(\mf{w})$. In particular, the sets $\mc{I}^0(\mf{w})$ provide a partition of $\C^{(0)}$.
\end{enumerate}
\end{lem}
\begin{proof}
We begin with part~(1). If $\mf{u}\in\mc{I}(\mf{w})$, the carrier of $\mf{w}$ is contained in the carrier of $\mf{u}$, and it is clear that $\mf{u}$ is transverse to all other hyperplanes that intersect the carrier of $\mf{w}$ (i.e.\ $\mf{w}$ and the hyperplanes transverse to $\mf{w}$ other than $\mf{u}$). Conversely, suppose that $\mf{u}\not\in\mc{I}(\mf{w})$. Then there exists $v\in X^{(0)}$ adjacent to $\mf{w}$, but not to $\mf{u}$. There exists $\mf{v}\in\mscr{W}_v$ separating $\mf{u}$ and $v$. If $\mf{u}$ is transverse to $\mf{w}$, so must be $\mf{v}$. In conclusion, either $\mf{u}$ is not transverse to $\mf{w}$, or $\mf{v}$ is transverse to $\mf{w}$ and $\mf{u}$ is not transverse to $\mf{v}$. 

Finally, $\#\mc{I}(\mf{w})\leq\dim X$ follows from the observation that the elements of $\mc{I}(\mf{w})$ are pairwise transverse. This completes the proof of part~(1).

Recall that we have $\mf{u}\in\mc{I}^0(\mf{w})$ if and only if $\mf{u}$ and $\mf{w}$ have the same carrier. Part~(3) is immediate from the fact that this is an equivalence relation. Part~(2) follows from the additional observation that edges of $\C$ join exactly those pairs of hyperplanes that have intersecting carriers.
\end{proof}

\begin{cor}\label{the kernel}
Let $X$ be uniformly locally finite, with no extremal vertices. Let $N\leq S(\C^{(0)})$ be the subgroup leaving each subset $\mc{I}^0(\mf{w})\cu\C$ invariant.
\begin{enumerate}
\item $N$ is contained in $\aut\C$.
\item $N$ coincides with the kernel of $\rho\colon\aut\C\ra S(X^{(0)})$. 
\item $N$ is isomorphic to the direct product of the permutation groups $S(\mc{I}^0(\mf{w}))$. In particular, $\rho$ is injective if and only if each $\mc{I}^0(\mf{w})$ is a singleton.
\end{enumerate}
\end{cor}
\begin{proof}
Part~(1) is immediate from part~(2) of Lemma~\ref{I(w) lemma}. The first half of part~(3) is immediate from part~(3) of Lemma~\ref{I(w) lemma}. The second half of part~(3) will follow from part~(2).

Let us then conclude by proving part~(2). Recall that an element $\phi\in\aut\C$ lies in $\ker\rho$ if and only if $\phi$ leaves invariant each maximal clique of $\C$. Again by part~(2) of Lemma~\ref{I(w) lemma}, every maximal clique in $\C$ is a union of sets of the form $\mc{I}^0(\mf{w})$. Hence $N\leq\ker\rho$.

Conversely, consider $\phi\in\ker\rho$. Given that each $\mc{I}(\mf{w})$ is an intersection of maximal cliques, and $\phi$ leaves every maximal clique invariant, we have $\phi(\mc{I}(\mf{w}))=\mc{I}(\mf{w})$ for all $\mf{w}\in\mscr{W}(X)$. Recalling that $\phi(\mc{I}(\mf{w}))=\mc{I}(\phi(\mf{w}))$ by Remark~\ref{I(w) and automorphisms}, we have $\mc{I}(\mf{w})=\mc{I}(\phi(\mf{w}))$, hence $\phi(\mf{w})\in\mc{I}^0(\mf{w})$. Thus, part~(3) of Lemma~\ref{I(w) lemma} yields $\mc{I}^0(\phi(\mf{w}))=\mc{I}^0(\mf{w})$ and, again by Remark~\ref{I(w) and automorphisms}, we have $\phi(\mc{I}^0(\mf{w}))=\mc{I}^0(\mf{w})$ for every $\mf{w}\in\mscr{W}(X)$. Hence $\phi\in N$ and $\ker\rho\leq N$.
\end{proof}

Finally, we discuss when the homomorphism $\rho$ takes values within $\aut X$.

\begin{lem}\label{extremal vertices of hyperplanes}
Suppose that no hyperplane of $X$ has extremal vertices. Then: 
\begin{enumerate}
\item given a vertex $v\in X^{(0)}$ and transverse hyperplanes $\mf{u},\mf{w}\in\mscr{W}_v$, there exists $\mf{v}\in\mscr{W}_v-\{\mf{u},\mf{w}\}$ that is transverse to $\mf{u}$, but not to $\mf{w}$;
\item a vertex $w\in X^{(0)}$ is adjacent to $v\in X^{(0)}$ if and only if there does not exist $x\in X^{(0)}-\{v,w\}$ with $\mscr{W}_v\cap\mscr{W}_w\cu\mscr{W}_v\cap\mscr{W}_x$.
\end{enumerate}
\end{lem}
\begin{proof}
We first prove part~(1). Let $v'\in\mf{u}$ be the projection of $v$ to $\mf{u}$ (i.e.\ the midpoint of the only edge of $X$ that contains $v$ and crosses $\mf{u}$).
This is a vertex of the cubical structure on $\mf{u}$ and $\mf{u}\cap\mf{w}$ is a hyperplane of $\mf{u}$ adjacent to $v'$. Since $\mf{u}$ has no extremal vertices, there exists a hyperplane $\mf{v}'$ of the cube complex $\mf{u}$ that is adjacent to $v'$ and disjoint from $\mf{u}\cap\mf{w}$. If $\mf{v}\in\mscr{W}(X)$ is the hyperplane with $\mf{v}'=\mf{v}\cap\mf{u}$, then $\mf{v}$ is adjacent to $v$, transverse to $\mf{u}$, and disjoint from $\mf{w}$ (disjointness follows, for instance, from the Helly property for hyperplane carriers). 

We now prove part~(2). If $w$ is not adjacent to $v$, there exists a vertex $x\in X^{(0)}-\{v,w\}$ that is adjacent to $v$ and lies on a combinatorial geodesic between $v$ and $w$. By convexity of carriers, we have $\mscr{W}_v\cap\mscr{W}_w\cu\mscr{W}_v\cap\mscr{W}_x$.

Conversely, suppose that $v$ and $w$ are adjacent and let $x\in X^{(0)}-\{v,w\}$ be such that $\mscr{W}_v\cap\mscr{W}_w\cu\mscr{W}_v\cap\mscr{W}_x$. Let $\mf{w}$ be the only hyperplane separating $v$ and $w$; since $\mf{w}\in\mscr{W}_v\cap\mscr{W}_w$, the vertex $x$ must lie in the carrier of $\mf{w}$. Let $x'$ and $v'=w'$ be the projections of the vertices $x,v,w$ to the hyperplane $\mf{w}$; since $x\not\in\{v,w\}$, we have $x'\neq v'$. Hence there exists a hyperplane $\mf{u}'$ of the cube complex $\mf{w}$ such that $\mf{u}'$ is adjacent to $v'$ and separates $v'$ from $x'$. Since $\mf{w}$ has no extremal vertices, there exists a hyperplane $\mf{v}'$ of $\mf{w}$ such that $\mf{v}'$ is adjacent to $v'$ and disjoint from $\mf{u}'$; in particular, $\mf{v}'$ is not adjacent to $x'$. Now, let $\mf{v}\in\mscr{W}(X)$ be the hyperplane with $\mf{v}'=\mf{v}\cap\mf{w}$. Note that $x$ is not adjacent to $\mf{v}$, or $x$ would lie in a square of $X$ crossed by $\mf{v}$ and $\mf{w}$, in which case $x'$ would be adjacent to $\mf{v}'$. In conclusion, $\mf{v}\in\mscr{W}_v\cap\mscr{W}_w-\mscr{W}_x$, which contradicts our assumption that $\mscr{W}_v\cap\mscr{W}_w\cu\mscr{W}_v\cap\mscr{W}_x$. 
\end{proof}

\begin{thm}\label{isomorphism}
Let $X$ be a uniformly locally finite $\CAT$ cube complex with no extremal vertices and with no hyperplanes containing extremal vertices. The map $\iota\colon\aut X\ra\aut\C$ is an isomorphism and $\rho$ is its inverse.
\end{thm}
\begin{proof}
By part~(2) of Lemma~\ref{extremal vertices of hyperplanes}, any permutation of $X^{(0)}$ in the image of $\rho\colon\aut\C\ra S(X^{(0)})$ preserves adjacency of vertices. It follows that $\rho$ takes values in $\aut X$ and we have already shown in Corollary~\ref{the map rho} that $\rho\o\iota=\id_{\aut X}$. Finally, part~(1) of Lemma~\ref{extremal vertices of hyperplanes} and part~(1) of Lemma~\ref{I(w) lemma} guarantee that $\mc{I}(\mf{w})=\{\mf{w}\}$ for every $\mf{w}\in\mscr{W}(X)$. Thus, Corollary~\ref{the kernel} shows that $\rho$ is injective.
\end{proof}

\section{Proof of the \protect\hyperlink{Davis intro}{Proposition}.}

In this section, we consider a right-angled Coxeter group $W_{\G}$, the universal cover $Y_{\G}$ of its Davis complex, and the contact graph $\C_{\G}=\C(Y_{\G})$. 
 
Let ${\rm cl}_n$ denote the complete graph on $n$ vertices. We can always split the finite graph $\G$ as a join ${\rm cl}_n\ast\G'$ for some $n\geq 0$ and a subgraph $\G'\cu\G$ that is not a cone. This corresponds to splittings $W_{\G}=(\Z/2\Z)^n\x W_{\G'}$ and:
\begin{align*}
&Y_{\G}=[0,1]^n\x Y_{\G'},  & &\aut Y_{\G}=((\Z/2\Z)^n\rtimes S_n)\x\aut Y_{\G'}, \\
&\C_{\G}={\rm cl}_n\ast\C_{\G'}, & &\aut\C_{\G}=S_n\x\aut\C_{\G'},
\end{align*}  
where $S_n$ denotes the symmetric group on $n$ elements. The natural map $\iota\colon\aut Y_{\G}\ra\aut\C_{\G}$ vanishes on the $(\Z/2\Z)^n$ subgroup, and it restricts to the natural map $\iota\colon\aut Y_{\G'}\ra\aut\C_{\G'}$ on the right-hand factors. 

Therefore, the general study of $\C_{\G}$ and its automorphism group reduces to the case $n=0$. Since links of vertices of $Y_{\G}$ are all isomorphic to the graph $\G$, this is equivalent to $Y_{\G}$ having no extremal vertices. 

Let us write $Y=Y_{\G}$ and $\C=\C_{\G}$ for short in the rest of the section. The discussion on Salvetti complexes in the paragraph before Remark~\ref{non-discrete} readily generalises to Davis complexes, showing that every hyperplane of $Y$ is labelled by a vertex of $\G$. Let $\g\colon\mscr{W}(Y)\ra\G^{(0)}$ denote the map assigning to each hyperplane its label. 

\begin{lem}\label{hyperplanes coming from star inclusions}
Given $\mf{w},\mf{u}\in\mscr{W}(Y)$, we have $\mf{u}\in\mc{I}(\mf{w})$ if and only if the carriers of $\mf{u}$ and $\mf{w}$ intersect and $\St\g(\mf{w})\cu\St\g(\mf{u})$. In this case, we have $\mf{u}\in\mc{I}^0(\mf{w})$ if and only if $\St\g(\mf{w})=\St\g(\mf{u})$.
\end{lem}
\begin{proof}
If $\mf{u}\in\mc{I}(\mf{w})$, the carrier of $\mf{w}$ is contained in the carrier of $\mf{u}$, and we have $\St\g(\mf{w})\cu\St\g(\mf{u})$ by part~(1) of Lemma~\ref{I(w) lemma}. Conversely, suppose that $\St\g(\mf{w})\cu\St\g(\mf{u})$ and that a vertex $v\in Y^{(0)}$ lies in the carrier of both $\mf{u}$ and $\mf{w}$. Any other vertex $w$ in the carrier of $\mf{w}$ is joined to $v$ by a path that only crosses edges labelled by elements of $\St\g(\mf{w})$. Since $\St\g(\mf{w})\cu\St\g(\mf{u})$, none of these edges can leave the carrier of $\mf{u}$, hence $\mf{u}\in\mscr{W}_w$. This shows that $\mf{u}\in\mc{I}(\mf{w})$. The statement about $\mc{I}^0(\mf{w})$ follows immediately.
\end{proof}

We are now ready to prove the \hyperlink{Davis intro}{Proposition} from the introduction, which we recall here for ease of reference.

\begin{prop}\label{prop recalled}
Let $W_{\G}$ have no finite direct factors. Then:
\begin{enumerate}
\item $Y$ has no extremal vertices, so $\iota\colon\aut Y\ra\aut\C$ is injective.
\item $Y$ has a hyperplane with extremal vertices if and only if there exist distinct vertices $a,b\in\G^{(0)}$ with $\St a\cu\St b$. In this case, the subgroup $\iota(\aut Y)<\aut\C$ has infinite index.
\item The homomorphism $\rho\colon\aut\C\ra S(Y^{(0)})$ is injective if and only if there do not exist vertices $a,b\in\G^{(0)}$ with $\St a=\St b$. When $\rho$ is not injective, its kernel is an uncountable torsion subgroup.
\end{enumerate}
\end{prop}
\begin{proof}
As already observed, the fact that $W_{\G}$ has no finite factors implies that $Y$ has no extremal vertices. Part~(1) thus follows from Corollary~\ref{the map rho}. 

Let us now prove part~(3). By part~(1) of Lemma~\ref{I(w) lemma}, the cardinality of each set $\#\mc{I}^0(\mf{w})$ is bounded above by $\dim Y$. Thus, part~(3) of Corollary~\ref{the kernel} shows that $\ker\rho$ is a (possibly trivial) torsion group. If no two vertices of $\G$ have the same star, then Lemma~\ref{hyperplanes coming from star inclusions} implies that each set $\mc{I}^0(\mf{w})$ is a singleton. In this case, part~(3) of Corollary~\ref{the kernel} shows that $\rho$ is injective. 

Conversely, suppose that distinct vertices $a,b\in\G^{(0)}$ have the same star. Since $W_{\G}$ has no direct factors, there exist infinitely many $\mf{w}_i\in\mscr{W}(Y)$ with $\g(\mf{w}_i)=a$. The sets $\mc{I}^0(\mf{w}_i)$ are pairwise disjoint and each contains at least two elements, by Lemma~\ref{hyperplanes coming from star inclusions}. Part~(3) of Corollary~\ref{the kernel} yields that $\ker\rho$ is a direct product of countably many non-trivial groups, hence uncountable.

We conclude the proof of the proposition by addressing part~(2). For every $\mf{w}\in\mscr{W}(Y)$, the induced cubical structure on $\mf{w}$ is isomorphic to $Y_{\L}$, where $\L$ is the full subgraph of $\G$ with vertex set $\lk\g(\mf{w})$. This has extremal vertices if and only if $\L$ is a cone over some $b\in\L^{(0)}$, i.e.\ if $\St\g(\mf{w})\cu\St b$. We deduce that $Y$ contains a hyperplane with an extremal vertex if and only if there exist $a,b\in\G^{(0)}$ with $\St a\cu\St b$. 

We are only left to show the second half of part~(2). Let $a,b\in\G^{(0)}$ be distinct vertices with $\St a\cu\St b$, consider $w\in Y^{(0)}$, and let $\mf{a},\mf{b}\in\mscr{W}_w$ be labelled by $a,b$, respectively. Let $\mf{a}^{\pm}$ be the two halfspaces bounded by $\mf{a}$. We define a partition $\mscr{W}(Y)=\mc{A}^+\sqcup\mc{A}^-\sqcup\mc{T}$, where a hyperplane lies in $\mc{A}^{\pm}$ if it is contained in $\mf{a}^{\pm}$, and it lies in $\mc{T}$ if it is transverse or equal to $\mf{a}$.

Let $r_{\mf{b}}\in W_{\G}<\aut Y$ be the reflection in the hyperplane $\mf{b}$. Consider the map $\phi\colon\C^{(0)}\ra\C^{(0)}$ defined as the identity on $\mc{A}^-\sqcup\mc{T}$, and as $\iota(r_{\mf{b}})$ on $\mc{A}^+\sqcup\mc{T}$ (note that $\iota(r_{\mf{b}})$ coincides with the identity on $\mc{T}$). Since no edge of $\C$ connects an element of $\mc{A}^+$ to an element of $\mc{A}^-$, it is clear that $\phi\in\aut\C$. 

If a vertex $v\in Y^{(0)}$ lies in the halfspace $\mf{a}^-$, we have $\mscr{W}_v\cu\mc{A}^-\sqcup\mc{T}$, and this set is fixed pointwise by $\phi$. In this case, we have $\rho(\phi)v=v$. On the other hand, if $v$ lies in $\mf{a}^+$, we have $\mscr{W}_v\cu\mc{A}^+\sqcup\mc{T}$, where $\phi=\iota(r_{\mf{b}})$. Hence $\rho(\phi)v=r_{\mf{b}}v$. Looking at the action on any square of $Y$ that is crossed by both $\mf{a}$ and $\mf{b}$, we see that $\rho(\phi)\not\in\aut Y$.

Let us write $\phi_{\mf{a}^+,\mf{b}}$ in the rest of the proof, highlighting the dependence on the choice of $\mf{a}^+$ and $\mf{b}$ in the definition of $\phi$. Pick a sequence $\mf{a}_n\in\mscr{W}(Y)$ of hyperplanes with $\g(\mf{a}_n)=a$ such that the distance between $\mf{a}$ and $\mf{a}_n$ diverges\footnote{One can show the existence of the $\mf{a}_n$ as follows. Since $W_{\G}$ has no finite factors, there exists a hyperplane $\mf{c}\in\mscr{W}(Y)$ disjoint from $\mf{a}$. Let $r_{\mf{a}},r_{\mf{c}}\in W_{\G}$ denote the reflections in $\mf{a}$ and $\mf{c}$, respectively. Then we can take $\mf{a}_n:=(r_{\mf{c}}r_{\mf{a}})^n\mf{a}$.}. Let $\mf{b}_n\in\mscr{W}(Y)$ be (the only) hyperplanes transverse to $\mf{a}_n$ with $\g(\mf{b}_n)=b$. We choose the side $\mf{a}_n^+$ so that it is disjoint from $\mf{a}^+$ and set $\psi_n:=\phi_{\mf{a}^+,\mf{b}}\o\phi_{\mf{a}_n^+,\mf{b}_n}$. By the above discussion, we have $\psi_n\in\aut\C$. 

The permutation $\rho(\psi_n)\in S(Y^{(0)})$ fixes exactly those vertices of $Y$ that do not lie in $\mf{a}^+\sqcup\mf{a}_n^+$. Since the distance between $\mf{a}^+$ and $\mf{a}_n^+$ diverges, the $\rho(\psi_n)$ lie in pairwise distinct cosets of ${\aut Y<\rho(\aut\C)}$. Hence, since $\rho\o\iota=\id_{\aut Y}$, the $\psi_n$ lie in pairwise distinct cosets of ${\iota(\aut Y)<\aut\C}$. 

This concludes the proof.
\end{proof}

\bibliography{mybib}

\begin{thebibliography}{BKMM12}

\bibitem[BCM12]{ending-2}
Jeffrey~F. Brock, Richard~D. Canary, and Yair~N. Minsky.
\newblock The classification of {K}leinian surface groups, {II}: {T}he ending
  lamination conjecture.
\newblock {\em Ann. of Math. (2)}, 176(1):1--149, 2012.

\bibitem[BFIM19]{BFI}
Jonas Beyrer, Elia Fioravanti, and Merlin Incerti-Medici.
\newblock {$\rm CAT(0)$} cube complexes are determined by their boundary cross
  ratio.
\newblock {\em arXiv:1805.08478v4. To appear on Groups Geom. Dyn.}, 2019.

\bibitem[BHS17a]{HHS-asdim}
Jason Behrstock, Mark~F. Hagen, and Alessandro Sisto.
\newblock Asymptotic dimension and small-cancellation for hierarchically
  hyperbolic spaces and groups.
\newblock {\em Proc. Lond. Math. Soc. (3)}, 114(5):890--926, 2017.

\bibitem[BHS17b]{HHS1}
Jason Behrstock, Mark~F. Hagen, and Alessandro Sisto.
\newblock Hierarchically hyperbolic spaces, {I}: {C}urve complexes for cubical
  groups.
\newblock {\em Geom. Topol.}, 21(3):1731--1804, 2017.

\bibitem[BHS19a]{HHS2}
Jason Behrstock, Mark Hagen, and Alessandro Sisto.
\newblock Hierarchically hyperbolic spaces {II}: {C}ombination theorems and the
  distance formula.
\newblock {\em Pacific J. Math.}, 299(2):257--338, 2019.

\bibitem[BHS19b]{HHS-qf}
Jason Behrstock, Mark~F. Hagen, and Alessandro Sisto.
\newblock Quasiflats in hierarchically hyperbolic spaces.
\newblock {\em arXiv:1704.04271v2}, 2019.

\bibitem[BJN10]{BJN}
Jason~A. Behrstock, Tadeusz Januszkiewicz, and Walter~D. Neumann.
\newblock Quasi-isometric classification of some high dimensional right-angled
  {A}rtin groups.
\newblock {\em Groups Geom. Dyn.}, 4(4):681--692, 2010.

\bibitem[BKMM12]{BKMM}
Jason Behrstock, Bruce Kleiner, Yair Minsky, and Lee Mosher.
\newblock Geometry and rigidity of mapping class groups.
\newblock {\em Geom. Topol.}, 16(2):781--888, 2012.

\bibitem[BKS08]{BKS}
Mladen Bestvina, Bruce Kleiner, and Michah Sageev.
\newblock The asymptotic geometry of right-angled {A}rtin groups. {I}.
\newblock {\em Geom. Topol.}, 12(3):1653--1699, 2008.

\bibitem[BN08]{Behrstock-Neumann}
Jason~A. Behrstock and Walter~D. Neumann.
\newblock Quasi-isometric classification of graph manifold groups.
\newblock {\em Duke Math. J.}, 141(2):217--240, 2008.

\bibitem[CdlH16]{Cornulier-Harpe}
Yves Cornulier and Pierre de~la Harpe.
\newblock {\em Metric geometry of locally compact groups}, volume~25 of {\em
  EMS Tracts in Mathematics}.
\newblock European Mathematical Society (EMS), Z\"{u}rich, 2016.
\newblock Winner of the 2016 EMS Monograph Award.

\bibitem[DHS17]{HHS-bound}
Matthew~Gentry Durham, Mark~F. Hagen, and Alessandro Sisto.
\newblock Boundaries and automorphisms of hierarchically hyperbolic spaces.
\newblock {\em Geom. Topol.}, 21(6):3659--3758, 2017.

\bibitem[DHS18]{HHS-dehn}
Fran\c{c}ois Dahmani, Mark Hagen, and Alessandro Sisto.
\newblock Dehn filling {D}ehn twists.
\newblock {\em arXiv:1812.09715v1}, 2018.

\bibitem[Gen19]{Genevois-surv}
Anthony Genevois.
\newblock Hyperbolicities in {${\rm CAT}(0)$} cube complexes.
\newblock {\em arXiv:1709.08843v2}, 2019.

\bibitem[Hag14]{Hagen-contact}
Mark~F. Hagen.
\newblock Weak hyperbolicity of cube complexes and quasi-arboreal groups.
\newblock {\em J. Topol.}, 7(2):385--418, 2014.

\bibitem[Ham07]{Ham-QI}
Ursula Hamenst\"{a}dt.
\newblock Geometry of the mapping class groups {III}: {Q}uasi-isometric
  rigidity.
\newblock {\em arXiv:math/0512429}, 2007.

\bibitem[Har79]{Harvey1}
William~J. Harvey.
\newblock Geometric structure of surface mapping class groups.
\newblock In {\em Homological group theory ({P}roc. {S}ympos., {D}urham,
  1977)}, volume~36 of {\em London Math. Soc. Lecture Note Ser.}, pages
  255--269. Cambridge Univ. Press, Cambridge-New York, 1979.

\bibitem[Har81]{Harvey2}
William~J. Harvey.
\newblock Boundary structure of the modular group.
\newblock In {\em Riemann surfaces and related topics: {P}roceedings of the
  1978 {S}tony {B}rook {C}onference ({S}tate {U}niv. {N}ew {Y}ork, {S}tony
  {B}rook, {N}.{Y}., 1978)}, volume~97 of {\em Ann. of Math. Stud.}, pages
  245--251. Princeton Univ. Press, Princeton, N.J., 1981.

\bibitem[HK18]{Huang-Kleiner}
Jingyin Huang and Bruce Kleiner.
\newblock Groups quasi-isometric to right-angled {A}rtin groups.
\newblock {\em Duke Math. J.}, 167(3):537--602, 2018.

\bibitem[HP98]{Haglund-Paulin}
Fr\'{e}d\'{e}ric Haglund and Fr\'{e}d\'{e}ric Paulin.
\newblock Simplicit\'{e} de groupes d'automorphismes d'espaces \`a courbure
  n\'{e}gative.
\newblock In {\em The {E}pstein birthday schrift}, volume~1 of {\em Geom.
  Topol. Monogr.}, pages 181--248. Geom. Topol. Publ., Coventry, 1998.

\bibitem[HS18]{Hagen-Susse}
Mark~F. Hagen and Tim Susse.
\newblock On hierarchical hyperbolicity of cubical groups.
\newblock {\em arXiv:1609.01313v2}, 2018.

\bibitem[Hua17]{Huang-GT2}
Jingyin Huang.
\newblock Quasi-isometric classification of right-angled {A}rtin groups {I}:
  the finite out case.
\newblock {\em Geom. Topol.}, 21(6):3467--3537, 2017.

\bibitem[Hua18]{Huang-IM}
Jingyin Huang.
\newblock Commensurability of groups quasi-isometric to {RAAG}s.
\newblock {\em Invent. Math.}, 213(3):1179--1247, 2018.

\bibitem[HW08]{HW-GAFA}
Fr\'{e}d\'{e}ric Haglund and Daniel~T. Wise.
\newblock Special cube complexes.
\newblock {\em Geom. Funct. Anal.}, 17(5):1551--1620, 2008.

\bibitem[Iva97]{Ivanov1}
Nikolai~V. Ivanov.
\newblock Automorphism of complexes of curves and of {T}eichm\"{u}ller spaces.
\newblock {\em Internat. Math. Res. Notices}, (14):651--666, 1997.

\bibitem[Iva02]{Ivanov2}
Nikolai~V. Ivanov.
\newblock Mapping class groups.
\newblock In {\em Handbook of geometric topology}, pages 523--633.
  North-Holland, Amsterdam, 2002.

\bibitem[KK13]{KK13}
Sang-hyun Kim and Thomas Koberda.
\newblock Embedability between right-angled {A}rtin groups.
\newblock {\em Geom. Topol.}, 17(1):493--530, 2013.

\bibitem[KK14]{KK14}
Sang-Hyun Kim and Thomas Koberda.
\newblock The geometry of the curve graph of a right-angled {A}rtin group.
\newblock {\em Internat. J. Algebra Comput.}, 24(2):121--169, 2014.

\bibitem[Mar19]{Margolis-RAAG}
Alexander Margolis.
\newblock Quasi-isometry classification of {RAAG}s that split over cyclic
  subgroups.
\newblock {\em arXiv:1803.05493v2}, 2019.

\bibitem[Min10]{ending-1}
Yair Minsky.
\newblock The classification of {K}leinian surface groups. {I}. {M}odels and
  bounds.
\newblock {\em Ann. of Math. (2)}, 171(1):1--107, 2010.

\bibitem[MM99]{MM1}
Howard~A. Masur and Yair~N. Minsky.
\newblock Geometry of the complex of curves. {I}. {H}yperbolicity.
\newblock {\em Invent. Math.}, 138(1):103--149, 1999.

\bibitem[MM00]{MM2}
Howard~A. Masur and Yair~N. Minsky.
\newblock Geometry of the complex of curves. {II}. {H}ierarchical structure.
\newblock {\em Geom. Funct. Anal.}, 10(4):902--974, 2000.

\bibitem[Sag14]{Sag14}
Michah Sageev.
\newblock {$\rm CAT(0)$} cube complexes and groups.
\newblock In {\em Geometric group theory}, volume~21 of {\em IAS/Park City
  Math. Ser.}, pages 7--54. Amer. Math. Soc., Providence, RI, 2014.

\bibitem[Tit74]{Tits-build}
Jacques Tits.
\newblock {\em Buildings of spherical type and finite {BN}-pairs}.
\newblock Lecture Notes in Mathematics, Vol. 386. Springer-Verlag, Berlin-New
  York, 1974.

\bibitem[Wis12]{Wise-astra}
Daniel~T. Wise.
\newblock {\em From riches to raags: 3-manifolds, right-angled {A}rtin groups,
  and cubical geometry}, volume 117 of {\em CBMS Regional Conference Series in
  Mathematics}.
\newblock Published for the Conference Board of the Mathematical Sciences,
  Washington, DC; by the American Mathematical Society, Providence, RI, 2012.

\end{thebibliography}
\bibliographystyle{alpha}

\end{document}